\documentclass[preprint,review,12pt]{elsarticle}
\usepackage{amsmath,amssymb,amsfonts,amsthm}
\usepackage{stmaryrd}
\usepackage{float}
\usepackage{natbib}
\usepackage{hyperref}
\usepackage{array}
\usepackage{background}
\SetBgContents{}

\biboptions{numbers,sort&compress}

\begin{document}

\begin{frontmatter}

\title{Numerical Methods for the Bogoliubov-Tolmachev-Shirkov model in superconductivity theory\tnoteref{1}} \tnotetext[1]{ Email: zhihaoge@henu.edu.cn, liruihua\_henu@163.com, fax:+86-371-23881696.\\
*Corresponding\ author.}

\author{Zhihao G${\rm e^*}$,\ Ruihua Li
}

\address{ School of Mathematics and Statistics\ \& Institute of Applied Mathematics, Henan University, Kaifeng 475004, P.R. China
}

\begin{abstract}
In the work, the numerical methods are designed for the Bogoliubov-Tolmachev-Shirkov model in superconductivity theory. The numerical methods are novel and  effective to determine the critical transition temperature and approximate to the energy gap function of the above model. Finally, a numerical example confirming the theoretical results is presented.
\end{abstract}


\end{frontmatter}

\thispagestyle{empty}


\vskip 0.5cm
\baselineskip 16pt
\newtheorem{thm}{Theorem}[section]
\newtheorem{lma}{Lemma}[section]
\newtheorem{rem}{Remarks}[section]
\newtheorem{assu}{Assumption}[section]
\newtheorem{definition}{Definition}[section]
\renewcommand{\theequation}{\arabic{section}.\arabic{equation}}

\section{Introduction}
\setcounter{equation}{0}
In the Bardeen-Cooper-Schrieffer(BCS) quantum theory of superconductivity, the superconducting state is characterized by a positive gap function, $\Delta(\textbf{x})$, which is the solution of the BCS equation
\begin{equation}\begin{split}
\Delta(\textbf{x})=\int_{\Omega}d\textbf{y}K(\textbf{x}, \textbf{y})\varphi_{\beta}
(\textbf{y}, \Delta(\textbf{y})),\label{eq1.1}
\end{split}\end{equation}
where
\begin{equation}\begin{split}
\varphi_{\beta}(\textbf{y},\Delta(\textbf{y}))
=H_{\beta}(((\textbf{y})^{2}+\Delta^{2}(\textbf{y}))^{1/2})
\Delta(\textbf{y}),\label{eq1.2}
\end{split}\end{equation}
with
\begin{equation}\begin{split}
H_{\beta}(t)=\frac{\tanh(1/2\beta t)}{t}.\label{eq1.3}
\end{split}\end{equation}
Where $\Omega$ is a bouned region, $\beta$ is the inverse of the absolute temperature, $T\geq0$, $K(\textbf{x}, \textbf{y})=-V_{\textbf{x}\textbf{y}}$ the negative matrix elements of the interaction potential of electrons with wave vectors $\textbf{x}$, $\textbf{y}\in\mathbb{R}^{3}$, and $\textbf{x}^{2}=|\textbf{x}|^{2}$, where $V_{\textbf{x}\textbf{y}}$ is generally the sum of two term: the first term, positive, arise from the repulsive coulomb force, while the second one, negative, from the attractive phonon force. As for the physical solution of BCS model, some researchers have studied, such as \cite{Hugenholtz}, \cite{Van Hemmen}, \cite{Shuji Watanabe}, \cite{Abraham Freiji}, \cite{X.H.Zheng}, \cite{Christian Hainzl} and so on.

For simplicity, one often consider the BCS gap equation in one dimension:
\begin{equation}\begin{split}
\Delta(x)=\int_{I}K(x, y)\frac{\tanh((1/2T)\sqrt{y^{2}
+\Delta^{2}(y)})}{\sqrt{y^{2}+\Delta^{2}(y)}}\Delta(y)dy,\label{eq3.1}
\end{split}\end{equation}
where $I=[-a, a]$ is a finite interval, $T\geq0$ is the absolute temperature, $\Delta(x)$ is the energy gap function so that $\Delta(x)=0$ corresponds to the normal phase and $\Delta(x)\neq0$ corresponds to the superconducting phase,
the original BCS assumption was given that the interaction kernel $K(x, y)$ is positive throughout the cut-off range from the Fermi surface up to a level $a>0$, which implies that the attractive phonon interaction is everywhere dominant.

Recently,  under the case of the interaction kernel $K(\textbf{x}, \textbf{y})$ that
\begin{equation}\begin{split}
 K(\textbf{x},\textbf{y})>0, K(\textbf{x}, \textbf{y})\leq \sigma(\textbf{y}),\ \frac{\sigma(\textbf{x})}{\textbf{x}^{2}+1}\in L(\mathbb{R}^{3}),\label{eq141023-1}
\end{split}\end{equation}
Du and Yang in \cite{Du-Yang} give some theoretical results: the BCS equation (\ref{eq1.1}) has a positive gap solution $\Delta(\textbf{x})>0$, representing the occurrence of superconductivity, while for $T=1/\beta>1/\beta_{c}=T_{c}$, the only solution of (\ref{eq1.1}) is the trivial one, $\Delta(\textbf{x})\equiv 0$, indicating the dominance of the normal phase; also give  a numerical method by the Min-Min scheme and Max-Max scheme to determine a critical temperature $T_{c}>0$.

However, this assumption (\ref{eq141023-1}) is only a simplified one. In order to make the model more realistic, Bogoliubov, Tolmachev, and Shirkov in \cite{Bogoliubov} considered the model (\ref{eq3.1}) in which the interaction kernel function $K(x, y)$ is given by the form
\begin{equation}\begin{split}
K(x, y)=K_{phonon}(x, y)+K_{Coulomb}(x, y),\label{eq140715-4.1}
\end{split}\end{equation}
where
\begin{equation}\begin{split}
K_{phonon}(x, y)&\equiv\frac{K_{1}}{2}>0,|x|,|y|<a;\\
K_{phonon}(x, y)&=0\quad {\rm otherwise},\\
K_{Coulomb}(x, y)&\equiv-\frac{K_{2}}{2}<0,|x|,|y|<b;\\
K_{Coulomb}(x, y)&=0\quad {\rm otherwise},\label{eq140715-4.3}
\end{split}\end{equation}
and $K_1, K_2$ are constants, $a>0$ is normally taken to be the Debye energy, $a=\hbar\omega_{D}$, and $b>a$ is a cut-off energy for the range of the screened Coulomb repulsion.

Since the kernel of the Bogoliubov-Tolmachev-Shirkov model is not positive but alternating, so the numerical methods in \cite{Du-Yang} do not work. And as we have known, there exist no effective numerical methods to handle this case. So, to overcome the above difficulties, we will develop a new numerical method to deal with the above model in this work.

The paper is organized as follows. In section 2, we design the Min-Mixed scheme and Max-Mixed scheme to Bogoliubov-Tolmachev-Shirkov model. And we show that these approximations lead to two numerical critical temperatures $\tau^{'}_{c}$ and $\tau^{"}_{c}$, and $\tau^{'}_{c}\leq T_{c}\leq\tau^{"}_{c}$. And also, we prove that there exist $(u,v)_{m}$ and $(u,v)_{M}$ such that $(u,v)_{m}\leq(u,v)\leq(u,v)_{M}$. In section 3, we  give a numerical test confirming the theoretical numerical results, and we obtain some important and interesting physical phenomenon.

\section{Numerical Methods }
\setcounter{equation}{0}

For the Bogoliubov-Tolmachev-Shirkov model, physicists expect the existence of a unique transition temperature $T_{c}>0$ so that, when $T<T_{c}$, (\ref{eq3.1}) has a positive solution representing the superconducting phase, but when $T>T_c$, the only solution is the trivial zero solution, representing the normal phase. Besides, as $T\rightarrow T_c$, the positive solution goes to zero.

With this form of the interaction kernel reflecting the mixed interaction of the phonon attraction and the Coulomb repulsion, one seeks(see \cite{Bogoliubov}\cite{Khalatnikov}\cite{Rickayzen}) a piecewise constant solution of the form
\begin{equation}\begin{split}
\Delta(x)&=\Delta_{1},|x|<a;\\
\Delta(x)&=\Delta_{2},a<|x|<b;\\
\Delta(x)&=0\quad {\rm otherwise}.\label{eq140715-4.4}
\end{split}\end{equation}
Hence, using (\ref{eq3.1}), (\ref{eq140715-4.1}), (\ref{eq140715-4.3}) and (\ref{eq140715-4.4}), we arrive at the coupled system
\begin{equation}\begin{split}
\Delta_1&=(K_{1}-K_{2})A_{\beta}(\Delta_{1})-K_{2}B_{\beta}(\Delta_{2}),\\
\Delta_2&=-K_{2}(A_{\beta}(\Delta_{1})+B_{\beta}(\Delta_{2})),\label{eq140715-4.5}
\end{split}\end{equation}
where $A_{\beta}$ and $B_{\beta}$ are the nonlinear transformations defined by
\begin{equation}\begin{split}
A_{\beta}(\Delta)&=\Delta\int_{0}^{a}f_{\beta}(\sqrt{\Delta^{2}+x^{2}})dx
=\Delta\int_{0}^{a}\frac{\tanh(1/2\beta\sqrt{\Delta^{2}+x^{2}})}{\sqrt{\Delta^{2}+x^{2}}}dx,\\
B_{\beta}(\Delta)&=\Delta\int_{a}^{b}f_{\beta}(\sqrt{\Delta^{2}+x^{2}})dx
=\Delta\int_{a}^{b}
\frac{\tanh(1/2\beta\sqrt{\Delta^{2}+x^{2}})}{\sqrt{\Delta^{2}+x^{2}}}dx.\label{eq140715-4.6}
\end{split}\end{equation}

The normal phase is characterized by the trivial solution of (\ref{eq140715-4.5}): $\Delta_{1}=0,\ \Delta_{2}=0$, and the superconducting phase is characterized by any nontrivial solution of (\ref{eq140715-4.5}) of the form
\begin{equation}\begin{split}
\Delta_{1}>0,\ \Delta_{2}<0.\label{eq140715-4.7}
\end{split}\end{equation}

And a rigorously  superconducting-normal phase transition theorem for the phonon-Coulomb interaction model of Bogoliubov-Tolmachev-Shirkov within the BCS theory has been established in \cite{Yisong Yang}:

\begin{thm}
There exists a unique and positive transition temperature, $T_{c}=1/\beta_{c}$, so that when $T<T_{c}$, the system (\ref{eq140715-4.5}) has a nontrivial solution of the form (\ref{eq140715-4.7}), and, when $T>T_{c}$, the only solution of (\ref{eq140715-4.5}) is the trivial solution, $\Delta_{1}=\Delta_{2}=0$.
\end{thm}


Next, we design a numerical method to determine the critical temperature. For convenience, introducing the new variables $u=\Delta_{1}$ and $v=-\Delta_{2}$, and using (\ref{eq140715-4.5}), we have
\begin{equation}\begin{split}
u&=(K_{1}-K_{2})A_{\beta}(u)+K_{2}B_{\beta}(v),\\
v&=K_{2}A_{\beta}(u)-K_{2}B_{\beta}(v).\label{eq140715-4.8}
\end{split}\end{equation}

It is seen that the superconducting phase is given by any positive solution of (\ref{eq140715-4.8}): $u>0,\ v>0$.

Observing the structure of $A_{\beta}$ and $B_{\beta}$ of (\ref{eq140715-4.6}), it is difficult to solve this equations (\ref{eq140715-4.8}) directly. So, in next discussion, we introduce two discretized versions, called the min-mixed and max-mixed approximations.

Now, we first introduce a partition of the interval $I$ as follows. Let $\{I_{j}|1\leq j\leq n\}$ be a collection of open subsets of $I$ such that
$$
I_{j}\cap I_{k}=\phi\quad (j\neq k),\quad \cup_{j=1}^{n}\bar{I}_{j}\supset I.
$$

To give the numerical methods for the model (\ref{eq140715-4.8}), we do with the problem by the following two cases.

{\bf Case I: $K_{1}>K_{2}$}.


In order to design the numerical scheme, we firstly introduce a definition.
\begin{definition}
We say that the pair $(u,v)$ is positive (nonnegative), if $u>0, v>0(u\geq0, v\geq0)$. Besides, we say $(u, v)>(u^{'}, v^{'})((u, v)\geq(u^{'}, v^{'}))$ if $(u-u^{'}>0, v-v^{'}>0)((u-u^{'}\geq0, v-v^{'}\geq0))$. We use the notation
$$
\chi=\{(u, v)\in \mathbb{R}\times\mathbb{R} | u\geq0, v\geq0\}.
$$
\end{definition}

\subsection{Min-Mixed and Max-Mixed  schemes}

The discrete scheme of the Bogoliubov-Tolmachev-Shirkov model below is
\begin{equation}\begin{split}
u=(K_{1}-K_{2})u\sum_{k}^{N}\min_{x\in\overline{I}_{k}}f_{\beta}(\sqrt{u^{2}+x^{2}})|\overline{I}_{k}|\\+
K_{2}v\sum_{k}^{N}\min_{x\in\overline{I^{'}}_{k}}f_{\beta}(\sqrt{v^{2}+x^{2}})|\overline{I^{'}}_{k}|,\\
v+K_{2}v\sum_{k}^{N}\max_{x\in\overline{I^{'}}_{k}}f_{\beta}(\sqrt{v^{2}+x^{2}})|\overline{I^{'}}_{k}|\\
=K_{2}u\sum_{k}^{N}\min_{x\in\overline{I}_{k}}f_{\beta}(\sqrt{u^{2}+x^{2}})|\overline{I}_{k}|.\label{eq141216-2.7}
\end{split}\end{equation}
We next will show that (\ref{eq141216-2.7}) has a positive solution if and only if it has a subsolution $(u_{0},v_{0})$ satisfying $u_0>0,\ v_0\geq0$ and
\begin{equation}\begin{split}
u_{0}\leq(K_{1}-K_{2})u_{0}\sum_{k}^{N}\min_{x\in\overline{I}_{k}}f_{\beta}(\sqrt{u_{0}^{2}+x^{2}})|\overline{I}_{k}|\\+
K_{2}v_{0}\sum_{k}^{N}\min_{x\in\overline{I^{'}}_{k}}f_{\beta}(\sqrt{v_{0}^{2}+x^{2}})|\overline{I^{'}}_{k}|,\\
v_{0}+K_{2}v_{0}\sum_{k}^{N}\max_{x\in\overline{I^{'}}_{k}}f_{\beta}(\sqrt{v_{0}^{2}+x^{2}})|\overline{I^{'}}_{k}|\\
\leq K_{2}u_{0}\sum_{k}^{N}\min_{x\in\overline{I}_{k}}f_{\beta}(\sqrt{u_{0}^{2}+x^{2}})|\overline{I}_{k}|.\label{eq141217-2.7}
\end{split}\end{equation}
To this end, we first define the iterative scheme
\begin{equation}\begin{split}
u_{n+1}=(K_{1}-K_{2})u_{n+1}\sum_{k}^{N}\min_{x\in\overline{I}_{k}}f_{\beta}(\sqrt{u_{n+1}^{2}+x^{2}})|\overline{I}_{k}|\\+
K_{2}v_{n}\sum_{k}^{N}\min_{x\in\overline{I^{'}}_{k}}f_{\beta}(\sqrt{v_{n}^{2}+x^{2}})|\overline{I^{'}}_{k}|,\\
v_{n+1}+K_{2}v_{n+1}\sum_{k}^{N}\max_{x\in\overline{I^{'}}_{k}}f_{\beta}(\sqrt{v_{n+1}^{2}+x^{2}})|\overline{I^{'}}_{k}|\\
=K_{2}u_{n+1}\sum_{k}^{N}\min_{x\in\overline{I}_{k}}f_{\beta}(\sqrt{u_{n+1}^{2}+x^{2}})|\overline{I}_{k}|,\\
n=1,2,\ldots; \ v_{1}=v_{0}.\label{eq140721-4.12}
\end{split}\end{equation}
The solution of (\ref{eq141216-2.7}) is denoted by $(u,v)_{m}$, and (\ref{eq141216-2.7}) is called by  the Min-Mixed scheme.\\
The discrete scheme of the Bogoliubov-Tolmachev-Shirkov model up  is
\begin{equation}\begin{split}
u=(K_{1}-K_{2})u\sum_{k}^{N}\max_{x\in\overline{I}_{k}}f_{\beta}(\sqrt{u^{2}+x^{2}})|\overline{I}_{k}|\\
+K_{2}v\sum_{k}^{N}\max_{x\in\overline{I^{'}}_{k}}f_{\beta}(\sqrt{v^{2}+x^{2}})|\overline{I^{'}}_{k}|,\\
v+K_{2}v\sum_{k}^{N}\min_{x\in\overline{I^{'}}_{k}}f_{\beta}(\sqrt{v^{2}+x^{2}})|\overline{I^{'}}_{k}|\\
=K_{2}u\sum_{k}^{N}\max_{x\in\overline{I}_{k}}f_{\beta}(\sqrt{u^{2}+x^{2}})|\overline{I}_{k}|.\label{eq141216-2.9}
\end{split}\end{equation}
In fact, $(u_{0},v_{0})$ is also a subsolution of (\ref{eq141216-2.9}), namely,
\begin{equation}\begin{split}
u_{0}\leq(K_{1}-K_{2})u_{0}\sum_{k}^{N}\max_{x\in\overline{I}_{k}}f_{\beta}(\sqrt{u_{0}^{2}+x^{2}})|\overline{I}_{k}|\\+
K_{2}v_{0}\sum_{k}^{N}\max_{x\in\overline{I^{'}}_{k}}f_{\beta}(\sqrt{v_{0}^{2}+x^{2}})|\overline{I^{'}}_{k}|,\\
v_{0}+K_{2}v_{0}\sum_{k}^{N}\min_{x\in\overline{I^{'}}_{k}}f_{\beta}(\sqrt{v_{0}^{2}+x^{2}})|\overline{I^{'}}_{k}|\\
\leq K_{2}u_{0}\sum_{k}^{N}\max_{x\in\overline{I}_{k}}f_{\beta}(\sqrt{u_{0}^{2}+x^{2}})|\overline{I}_{k}|.\label{eq141218-2.7}
\end{split}\end{equation}
And the iterative scheme  is defined by
\begin{equation}\begin{split}
u_{n+1}=(K_{1}-K_{2})u_{n+1}\sum_{k}^{N}\max_{x\in\overline{I}_{k}}f_{\beta}(\sqrt{u_{n+1}^{2}+x^{2}})|\overline{I}_{k}|\\
+K_{2}v_{n}\sum_{k}^{N}\max_{x\in\overline{I^{'}}_{k}}f_{\beta}(\sqrt{v_{n}^{2}+x^{2}})|\overline{I^{'}}_{k}|,\\
v_{n+1}+K_{2}v_{n+1}\sum_{k}^{N}\min_{x\in\overline{I^{'}}_{k}}f_{\beta}(\sqrt{v_{n+1}^{2}+x^{2}})|\overline{I^{'}}_{k}|\\
=K_{2}u_{n+1}\sum_{k}^{N}\max_{x\in\overline{I}_{k}}f_{\beta}(\sqrt{u_{n+1}^{2}+x^{2}})|\overline{I}_{k}|,\\
n=1,2,\ldots; \ v_{1}=v_{0}.\label{eq140719-4.14}
\end{split}\end{equation}
$(u,v)_{M}$ is used to denote the solution of (\ref{eq141216-2.9}), and (\ref{eq141216-2.9}) is called by the Max-Mixed scheme.
\begin{rem}
The Min-Mixed scheme and Max-Mixed scheme are different from the Min-Min scheme and Max-Max scheme of \cite{Du-Yang}: the problem in the work is a system, while the problem of \cite{Du-Yang} is a single equation; the discrete schemes are very different.
\end{rem}

In order to prove the numerical solutions of the discrete system (\ref{eq141216-2.7}) and (\ref{eq141216-2.9}), we need to give the following lemmas.
Denote
\begin{equation}\begin{split}
A_{h}(u)&=u\sum_{k}^{N}\min_{x\in\overline{I}_{k}}f_{\beta}(\sqrt{u^{2}+x^{2}})|\overline{I}_{k}|,\\
B_{h}(v)&=v\sum_{k}^{N}\min_{x\in\overline{I^{'}}_{k}}f_{\beta}(\sqrt{v^{2}+x^{2}})|\overline{I^{'}}_{k}|.
\end{split}\end{equation}
\begin{lma}\label{lma4.1}
$$H_{h}(u)=u-(K_{1}-K_{2})A_{h}(u)$$ is monotone about $u$.
\end{lma}
\begin{proof}
The proof is similar to that for the continuous case in \cite{Yisong Yang} and is skipped here.
\end{proof}
\begin{lma}\label{lma4.2}
When $\beta>0$ is small, the only solution of (\ref{eq141216-2.7}) is the zero solution.
\end{lma}
\begin{proof}
 This is because $$A_{h}(u)\leq\frac{1}{2}\beta au,$$ and $$B_{h}(v)\leq\frac{1}{2}\beta (b-a)v.$$
 Therefore, when $\beta$ is small, the only non-negative solution of (\ref{eq141216-2.7}) is the trivial solution $u=0,\ v=0$.
 \end{proof}
\begin{lma}\label{lma4.3}
When $\beta>0$ is sufficiently large, (\ref{eq141216-2.7}) has a subsolution $(u_0,v_0)$ as it is defined in (\ref{eq141217-2.7}).
\end{lma}
\begin{proof}
 Indeed, we may start from the simple BCS discrete equation
\begin{equation}\begin{split}
u=(K_{1}-K_{2})A_{h}(u),\label{eq140721-4.16}
\end{split}\end{equation}
which may be obtained by setting $v=0$ in the first equation in (\ref{eq141216-2.7}). When $\beta$ is large, (\ref{eq140721-4.16}) has  a positive solution, say $u_{0}$(see \cite{Du-Yang}). Let $v_{0}=0$. Then the pair $(u_{0},v_{0})$ satisfing (\ref{eq141217-2.7}) is a subsolution.

\end{proof}
\begin{lma}\label{lma4.4}
There is a $\delta_{0}>0$, so that for any $u^{0}\geq \delta_{0}$, $u^{0}$ is a supersolution of the first equation of (\ref{eq141216-2.7}) for any $v$, in the sense that:
\begin{equation}\begin{split}
u^{0}-(K_{1}-K_{2})A_{h}(u^{0})
\geq K_{2}B_{h}(v),\quad \forall v.\label{eq141202-2.13}
\end{split}\end{equation}
\end{lma}
\begin{proof}
Since the function $A_{h}(u)$, $B_{h}(v)$ are bounded uniformly with respect to the parameter $\beta$, so we have
$$A_{h}(u)\leq C,$$
$$B_{h}(v)\leq C.$$
for some absolute constant $C>0$,
there is an absolute constant $\delta_{0}>0$ so that
\begin{equation}\begin{split}
\delta_{0}-(K_{1}-K_{2})A_{h}(\delta_{0})
\geq K_{2}B_{h}(v),\quad \forall v.
\end{split}\end{equation}
then using Lemma\ref{lma4.1}, we can obtain if $u^{0}\geq \delta_{0}$, $u^{0}$ is a supersolution which satisfes (\ref{eq141202-2.13}).
\end{proof}

\begin{lma}\label{lma4.5}
The Min-Mixed interation scheme (\ref{eq141216-2.7}) has a positive solution if and only if there is a nontrivial subsolution $(u_0, v_0)$.
\end{lma}
\begin{proof}
Using Lemma \ref{lma4.4}, there is an absolute constant $u^{0}>0$ so that
\begin{equation}\begin{split}
u^{0}-(K_{1}-K_{2})A_{h}(u^{0})
\geq K_{2}B_{h}(v),\quad \forall v.\label{eq141121-2.16}
\end{split}\end{equation}
In the iterative scheme (\ref{eq140721-4.12}), if $v_{1}=v_{0}\geq 0$, then $u_{2}>0$ and $u_{0}\leq u_{2}\leq u^{0}$ by
\begin{equation}\begin{split}
u_{0}-(K_{1}-K_{2})A_{h}(u_{0})
\leq K_{2}B_{h}(v_{0})
\end{split}\end{equation}
and (\ref{eq141121-2.16}).
Since the function
\begin{equation}\begin{split}
J_{h}(v)=v+K_{2}v\sum_{k}^{N}\max_{x\in\overline{I^{'}}_{k}}f_{\beta}(\sqrt{v^{2}+x^{2}})|\overline{I^{'}}_{k}|,\label{eq141121-2.18}
\end{split}\end{equation}
strictly increases with $J_{h}(0)=0$ and $J_{h}(\infty)=\infty$, the equation
\begin{equation}\begin{split}
J_{h}(v)=s,\label{eq141121-2.19}
\end{split}\end{equation}
has a unique solution, say $v$, in $[0,\infty]$ for each $s\in [0,\infty]$ and $v$ increases as $s$ increases. Hence, in
(\ref{eq140721-4.12}), $v_{2}>0$ is well defined and $v_{2}\geq v_{1}=v_{0}$.
Assume that the inequalities
\begin{equation}\begin{split}
0<u_{0}=u_{1}\leq u_{2}\leq\ldots\leq u_{l}\leq u^{0},\label{eq141121-2.20}
\end{split}\end{equation}
\begin{equation}\begin{split}
0\leq v_{0}=v_{1}\leq v_{2}\leq\ldots\leq v_{l},\label{eq141121-2.21}
\end{split}\end{equation}
hold at some step $l$. Then, in view of (\ref{eq141121-2.20}) and (\ref{eq141121-2.21}), $u_{l}$ and $v_{l}$ satisfy
\begin{equation}\begin{split}
K_{2}B_{h}(v_{l-1})
\leq K_{2}B_{h}(v_{l}).\label{eq140910-4.27}
\end{split}\end{equation}
Hence we arrive at $u_{l+1}\geq u_{l}$ after comparing (\ref{eq140910-4.27}) with (\ref{eq141121-2.21}) and reviewing the definition of $u_{l+1}$.
Thus
\begin{equation}\begin{split}
K_{2}A_{h}(u_{l})
\leq K_{2}A_{h}(u_{l+1}).\label{eq140910-4.28}
\end{split}\end{equation}
Obviously, $v_{l}\leq v_{l+1}$ in view of (\ref{eq141121-2.18}). Of course, $u_{l+1}\leq u^{0}$
because $u^{0}$ has been chosen to be a (universal) supersolution (see (\ref{eq141121-2.16})).

Therefore, we have shown that (\ref{eq141121-2.20}) and (\ref{eq141121-2.21}) are valid in general.

The boundedness of the sequence $\{v_{n}\}$ follows from the boundedness of the sequence $\{u_{n}\}$
and the second equation in (\ref{eq140721-4.12}). In fact,
$$v_{n}\leq K_{2}A_{h}(u^{0}),\quad n=1,2,\ldots$$

Taking $n\rightarrow\infty$  in the scheme (\ref{eq140721-4.12}), we obtain a numerical solution pair $(u, v)_{m}$ of the
Bogoliubov-Tolmachev-Shirkov model.
\end{proof}

\begin{lma}\label{lma4.6}
Let $$\Lambda=\{\beta>0\ |\ \textrm{When N is sufficiently large},\ (\ref{eq141216-2.7})\ \textrm{has a positive solution pair}\},$$ and $$\beta_c^{'}=inf\{\beta\ |\ \beta\in\Lambda\},$$ then $\Lambda$ is connected and $\beta_c^{'}>0$. Moreover, we have the relations $(\beta_{c}^{'}, \infty)\subset\Lambda$ and $[0, \beta_{c}^{'})\bigcap\Lambda=\phi$.
\end{lma}

\begin{proof}
To see this, we show that, if $\beta\in\Lambda$, then $\beta+\varepsilon\in\Lambda$ for any $\varepsilon>0$.

In fact, for $\beta\in\Lambda$, let $(u,v)_{m}$ be a positive solution pair of (\ref{eq141216-2.7}). We rewrite (\ref{eq141216-2.7}) as
\begin{equation}\begin{split}
u=(K_{1}-K_{2})u\sum_{k}^{N}\min_{x\in\overline{I}_{k}}f_{\beta}(\sqrt{u^{2}+x^{2}})|\overline{I}_{k}|\\+
K_{2}v\sum_{k}^{N}\min_{x\in\overline{I^{'}}_{k}}f_{\beta}(\sqrt{v^{2}+x^{2}})|\overline{I^{'}}_{k}|,\\
v+K_{2}v\sum_{k}^{N}\max_{x\in\overline{I^{'}}_{k}}f_{\beta}(\sqrt{v^{2}+x^{2}})|\overline{I^{'}}_{k}|\\
=K_{2}u\sum_{k}^{N}\min_{x\in\overline{I}_{k}}f_{\beta}(\sqrt{u^{2}+x^{2}})|\overline{I}_{k}|.\label{eq140911-4.32}
\end{split}\end{equation}
Since $v>0$, we may choose $r\in(0,1)$ so that
\begin{equation}\begin{split}
B_{\beta+\varepsilon}(rv)=B_{\beta}(v).\label{eq140911-4.33}
\end{split}\end{equation}
However, from (\ref{eq140911-4.32}), we have
\begin{equation}\begin{split}
u<(K_{1}-K_{2})u\sum_{k}^{N}\min_{x\in\overline{I}_{k}}f_{\beta+\varepsilon}(\sqrt{u^{2}+x^{2}})|\overline{I}_{k}|\\+
K_{2}v\sum_{k}^{N}\min_{x\in\overline{I^{'}}_{k}}f_{\beta}(\sqrt{v^{2}+x^{2}})|\overline{I^{'}}_{k}|,\\
v+K_{2}v\sum_{k}^{N}\max_{x\in\overline{I^{'}}_{k}}f_{\beta}(\sqrt{v^{2}+x^{2}})|\overline{I^{'}}_{k}|\\
<K_{2}u\sum_{k}^{N}\min_{x\in\overline{I}_{k}}f_{\beta+\varepsilon}(\sqrt{u^{2}+x^{2}})|\overline{I}_{k}|.\label{eq140911-4.34}
\end{split}\end{equation}
Combining (\ref{eq140911-4.33}) and (\ref{eq140911-4.34}), we obtain
\begin{equation}\begin{split}
u\leq(K_{1}-K_{2})u\sum_{k}^{N}\min_{x\in\overline{I}_{k}}f_{\beta+\varepsilon}(\sqrt{u^{2}+x^{2}})|\overline{I}_{k}|\\+
K_{2}rv\sum_{k}^{N}\min_{x\in\overline{I^{'}}_{k}}f_{\beta+\varepsilon}(\sqrt{(rv)^{2}+x^{2}})|\overline{I^{'}}_{k}|,\\
rv+K_{2}rv\sum_{k}^{N}\max_{x\in\overline{I^{'}}_{k}}f_{\beta+\varepsilon}(\sqrt{(rv)^{2}+x^{2}})|\overline{I^{'}}_{k}|\\
\leq K_{2}u\sum_{k}^{N}\min_{x\in\overline{I}_{k}}f_{\beta+\varepsilon}(\sqrt{u^{2}+x^{2}})|\overline{I}_{k}|.\label{eq140911-4.35}
\end{split}\end{equation}
In other words, we have recovered (\ref{eq141217-2.7}) with $u_{0}=u, v_{0}=rv$, and $\beta$ being replaced by $\beta+\varepsilon$.
Consequently, $\beta+\varepsilon\in\Lambda$.
\end{proof}

Using Lemma \ref{lma4.1}-Lemma \ref{lma4.6}, we obtain the following important result:
\begin{thm}\label{thm4.1}
There exists a number $\beta^{'}_{c}>0$ so that (\ref{eq141216-2.7}) has a nontrivial solution: $u>0,v>0$, for any $\beta:\beta^{'}_{c}<\beta\leq\infty$, while for $\beta<\beta^{'}_c$, the only solution of (\ref{eq141216-2.7}) is the trivial one, $u=0,v=0$.
\end{thm}

\begin{rem}
From Theorem \ref{thm4.1}, we  do not know if the only solution of (\ref{eq141216-2.7}) is the zero solution when $\beta=\beta^{'}_{c}$. We guess that it is true (one can see  Fig. \ref{fig6} and Fig. \ref{fig9}), but we are not able to prove it.
\end{rem}

In fact, we can obtain another important theorem.
\begin{thm}\label{thm4.2}
There exists a number $\beta^{"}_{c}>0$ so that (\ref{eq141216-2.9}) has a nontrivial solution $u>0,v>0$ for any $\beta:\beta^{"}_{c}<\beta\leq\infty$, while for $\beta<\beta^{"}_c$, the only solution of (\ref{eq141216-2.9}) is the trivial one, $u=0,v=0$.
\end{thm}
\begin{proof}
Similar to Theorem \ref{thm4.1}, the proof of this theorem can be carried out.
\end{proof}

Additionally, let us see an interesting comparison theorem.
\begin{thm}\label{thm4.3}
Let $\beta_{c}$, $\beta^{'}_{c}$ and $\beta^{"}_{c}$ are the corresponding critical numbers of (\ref{eq140715-4.8}), (\ref{eq141216-2.7}) and (\ref{eq141216-2.9}), respectively. Then

$$\beta^{'}_{c}\geq\beta_{c}\geq\beta^{"}_{c}.$$\\
Besides, if $(u,v)_{m}$, $(u,v)_{M}$ are  solutions of (\ref{eq141216-2.7}) and (\ref{eq141216-2.9}) respectively, $(u,v)$ is the solution of (\ref{eq140715-4.8}),  then

$$(u,v)_{m}\leq(u,v)\leq(u,v)_{M}.$$

\end{thm}
\begin{proof}
 For $\beta>\beta^{'}_{c}$, let $(u,v)_{m}$ be a nontrivial solution of (\ref{eq141216-2.7}) in $\chi$. Then $(u,v)_{m}$ is a subsolution of (\ref{eq140715-4.8}). Thus $(u,v)_{m}\leq(u,v)$ which can be obtained by interating from $(u,v)_{m}$. Consequently, $\beta>\beta_{c}$. Clearly, $\beta^{'}_{c}\geq\beta_{c}$ and $(u,v)_{m}\leq(u,v)$.

Next, take $\beta>\beta_{c}$ and assume that $(u,v)$ is a nontrivial solution of (\ref{eq140715-4.8}) in $\chi$. Then $(u,v)$ is a subsolution of (\ref{eq141216-2.9}).
Thus $(u,v)\leq(u,v)_{M}$ ($(u,v)_{M}$ is a nontrivial solution of (\ref{eq141216-2.9}). Consequently, $\beta>\beta^{"}_{c}$. So $\beta_{c}\geq\beta^{"}_{c}$.

The proof of Theorem \ref{thm4.3} is completed.
\end{proof}

{\bf Case II:} $K_{1}\leq K_{2}$.
In this case, (\ref{eq140715-4.8}) has been rewritted as
\begin{equation}\begin{split}
u+(K_{2}-K_{1})A_{\beta}(u)=K_{2}B_{\beta}(v),\\
v+K_{2}B_{\beta}(v)=K_{2}A_{\beta}(u).\label{eq140715-4.13}
\end{split}\end{equation}

The  Min-Mixed scheme below is
\begin{equation}\begin{split}
u+(K_{2}-K_{1})u\sum_{k}^{N}\max_{x\in\overline{I}_{k}}f_{\beta}(\sqrt{u^{2}+x^{2}})|\overline{I}_{k}|\\
=K_{2}v\sum_{k}^{N}\min_{x\in\overline{I^{'}}_{k}}f_{\beta}(\sqrt{v^{2}+x^{2}})|\overline{I^{'}}_{k}|,\\
v+K_{2}v\sum_{k}^{N}\max_{x\in\overline{I^{'}}_{k}}f_{\beta}(\sqrt{v^{2}+x^{2}})|\overline{I^{'}}_{k}|\\
=K_{2}u\sum_{k}^{N}\min_{x\in\overline{I}_{k}}f_{\beta}(\sqrt{u^{2}+x^{2}})|\overline{I}_{k}|.\label{eq141216-2.30}
\end{split}\end{equation}
As before, we can show that the system (\ref{eq141216-2.30}) has a positive solution pair if and only if there exists a nontrivial subsolution, $(u_{0},v_{0})$, satisfying
\begin{equation}\begin{split}
u_{0}+(K_{2}-K_{1})u_{0}\sum_{k}^{N}\max_{x\in\overline{I}_{k}}f_{\beta}(\sqrt{u_{0}^{2}+x^{2}})|\overline{I}_{k}|\\
\leq K_{2}v_{0}\sum_{k}^{N}\min_{x\in\overline{I^{'}}_{k}}f_{\beta}(\sqrt{v_{0}^{2}+x^{2}})|\overline{I^{'}}_{k}|,\\
v_{0}+K_{2}v_{0}\sum_{k}^{N}\max_{x\in\overline{I^{'}}_{k}}f_{\beta}(\sqrt{v_{0}^{2}+x^{2}})|\overline{I^{'}}_{k}|\\
\leq K_{2}u\sum_{k}^{N}\min_{x\in\overline{I}_{k}}f_{\beta}(\sqrt{u_{0}^{2}+x^{2}})|\overline{I}_{k}|,\label{eq141218-2.30}
\end{split}\end{equation}
where $u_{0}$, $v_{0}$ are positive.\\
In fact, define the  Min-Mixed interation scheme below as
\begin{equation}\begin{split}
u_{n+1}+(K_{2}-K_{1})u_{n+1}\sum_{k}^{N}\max_{x\in\overline{I}_{k}}f_{\beta}(\sqrt{u_{n+1}^{2}+x^{2}})|\overline{I}_{k}|\\
=K_{2}v_{n}\sum_{k}^{N}\min_{x\in\overline{I^{'}}_{k}}f_{\beta}(\sqrt{v_{n}^{2}+x^{2}})|\overline{I^{'}}_{k}|,\\
v_{n+1}+K_{2}v_{n+1}\sum_{k}^{N}\max_{x\in\overline{I^{'}}_{k}}f_{\beta}(\sqrt{v_{n+1}^{2}+x^{2}})|\overline{I^{'}}_{k}|\\
=K_{2}u_{n+1}\sum_{k}^{N}\min_{x\in\overline{I}_{k}}f_{\beta}(\sqrt{u_{n+1}^{2}+x^{2}})|\overline{I}_{k}|,\\
n=1, 2, \ldots;\ v=v_{0}.\label{eq141216-4.12}
\end{split}\end{equation}
Using the monotonicity of the function $$P_{h}(u)=u+(K_{2}-K_{1})u\sum_{k}^{N}\max_{x\in\overline{I}_{k}}f_{\beta}(\sqrt{u^{2}+x^{2}})|\overline{I}_{k}|,$$ and $$Q_{h}(v)=v+K_{2}v\sum_{k}^{N}\max_{x\in\overline{I^{'}}_{k}}f_{\beta}(\sqrt{v^{2}+x^{2}})|\overline{I^{'}}_{k}|,$$ we see that the sequences ${u_{n}}$ and ${v_{n}}$ are well defined and that
\begin{equation}\begin{split}
u_{0}=u_{1}\leq u_{2}\leq\ldots \leq u_{n}\leq \ldots,\ v_{0}=v_{1}\leq v_{2}\leq\ldots \leq v_{n}\leq \ldots,
\end{split}\end{equation}
Since the function $$B_{h}(v)=v\sum_{k}^{N}\min_{x\in\overline{I^{'}}_{k}}f_{\beta}(\sqrt{v^{2}+x^{2}})|\overline{I^{'}}_{k}|,$$ and  $$A_{h}(u)=u\sum_{k}^{N}\min_{x\in\overline{I}_{k}}f_{\beta}(\sqrt{u^{2}+x^{2}})|\overline{I}_{k}|,$$ are bounded, it follows from (\ref{eq141216-4.12}) that ${u_{n}}$ and ${v_{n}}$ are bounded sequences, Taking the limit $ n\rightarrow \infty$ in (\ref{eq141216-4.12}), we see that $u=\lim_{n\rightarrow \infty}u_{n}$ and $v=\lim_{n\rightarrow \infty}v_{n}$ make a solution pair to the system (\ref{eq141216-2.30}).\\
The Max-Mixed scheme up  is
\begin{equation}\begin{split}
u+(K_{2}-K_{1})u\sum_{k}^{N}\min_{x\in\overline{I}_{k}}f_{\beta}(\sqrt{u^{2}+x^{2}})|\overline{I}_{k}|\\
=K_{2}v\sum_{k}^{N}\max_{x\in\overline{I^{'}}_{k}}f_{\beta}(\sqrt{v^{2}+x^{2}})|\overline{I^{'}}_{k}|,\\
v+K_{2}v\sum_{k}^{N}\min_{x\in\overline{I^{'}}_{k}}f_{\beta}(\sqrt{v_{n+1}^{2}+x^{2}})|\overline{I^{'}}_{k}|\\
=K_{2}u\sum_{k}^{N}\max_{x\in\overline{I}_{k}}f_{\beta}(\sqrt{u^{2}+x^{2}})|\overline{I}_{k}|,\label{eq141216-2.32}
\end{split}\end{equation}
Obviously, $(u_{0},v_{0})$ defined in (\ref{eq141218-2.30}) is also a subsolution of (\ref{eq141216-2.32}), and
the Max-Mixed interation scheme up  is
\begin{equation}\begin{split}
u_{n+1}+(K_{2}-K_{1})u_{n+1}\sum_{k}^{N}\min_{x\in\overline{I}_{k}}f_{\beta}(\sqrt{u_{n+1}^{2}+x^{2}})|\overline{I}_{k}|\\
=K_{2}v_{n}\sum_{k}^{N}\max_{x\in\overline{I^{'}}_{k}}f_{\beta}(\sqrt{v_{n}^{2}+x^{2}})|\overline{I^{'}}_{k}|,\\
v_{n+1}+K_{2}v_{n+1}\sum_{k}^{N}\min_{x\in\overline{I^{'}}_{k}}f_{\beta}(\sqrt{v_{n+1}^{2}+x^{2}})|\overline{I^{'}}_{k}|\\
=K_{2}u_{n+1}\sum_{k}^{N}\max_{x\in\overline{I}_{k}}f_{\beta}(\sqrt{u_{n+1}^{2}+x^{2}})|\overline{I}_{k}|,\\
n=1, 2, \ldots;\ v=v_{0}.\label{eq141216-4.14}
\end{split}\end{equation}
The convergence of (\ref{eq141216-4.14}) which similar to (\ref{eq141216-4.12}) will no longer be proved here.
And the choice of subsolution can reference \cite{Yisong Yang}. In section 3 , we shall present the numerical results.

\section{Numerical Test}
\setcounter{equation}{0}
In this section, we shall caculate specifically a example which corresponding to the above section.

{\bf Case 1}: $K_{1}>K_{2}$.
 Taking
\begin{equation}\begin{split}
a=1,\ b=1.5,\\
K_{1}=2, K_{2}=0.1.\\
\end{split}
\end{equation}

\begin{figure}[H]
\centering
\includegraphics[height=0.5\textwidth, width=\textwidth]{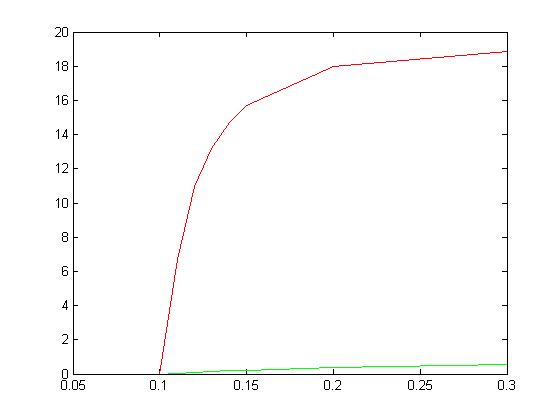}
\caption{Min-Mixed interation($a=1, \beta_{c}\approx0.1$), $N=50$, $(u(up), v(down))$ vs $\beta$}\label{fig6}
\end{figure}

\begin{figure}[H]
\centering
\includegraphics[height=0.5\textwidth, width=\textwidth]{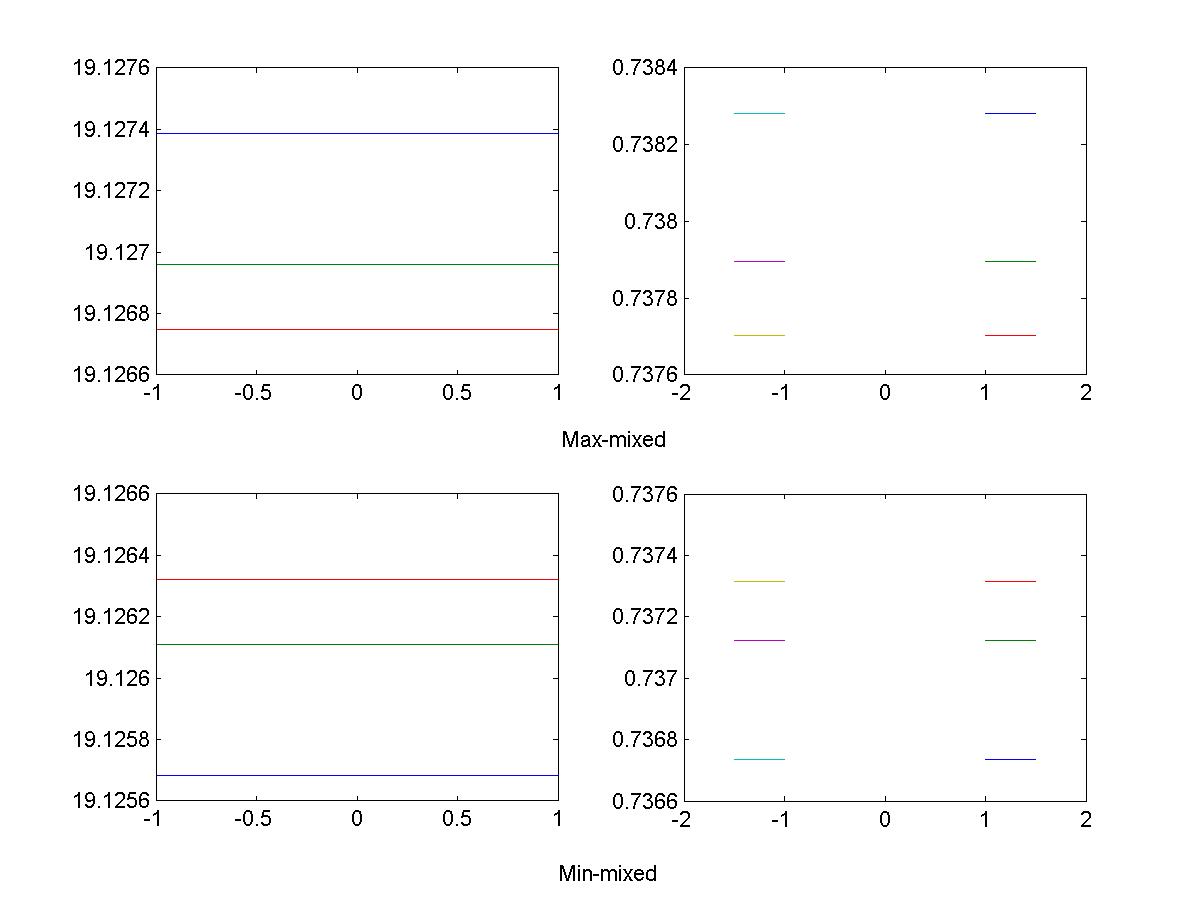}
\caption{Max-Mixed interation(from up to down), Min-Mixed interation(from bottom to up), $u(left),v(right),\beta=6, N=50, 100, 200$}\label{fig7}
\end{figure}

\begin{figure}[H]
\centering
\includegraphics[height=0.5\textwidth, width=\textwidth]{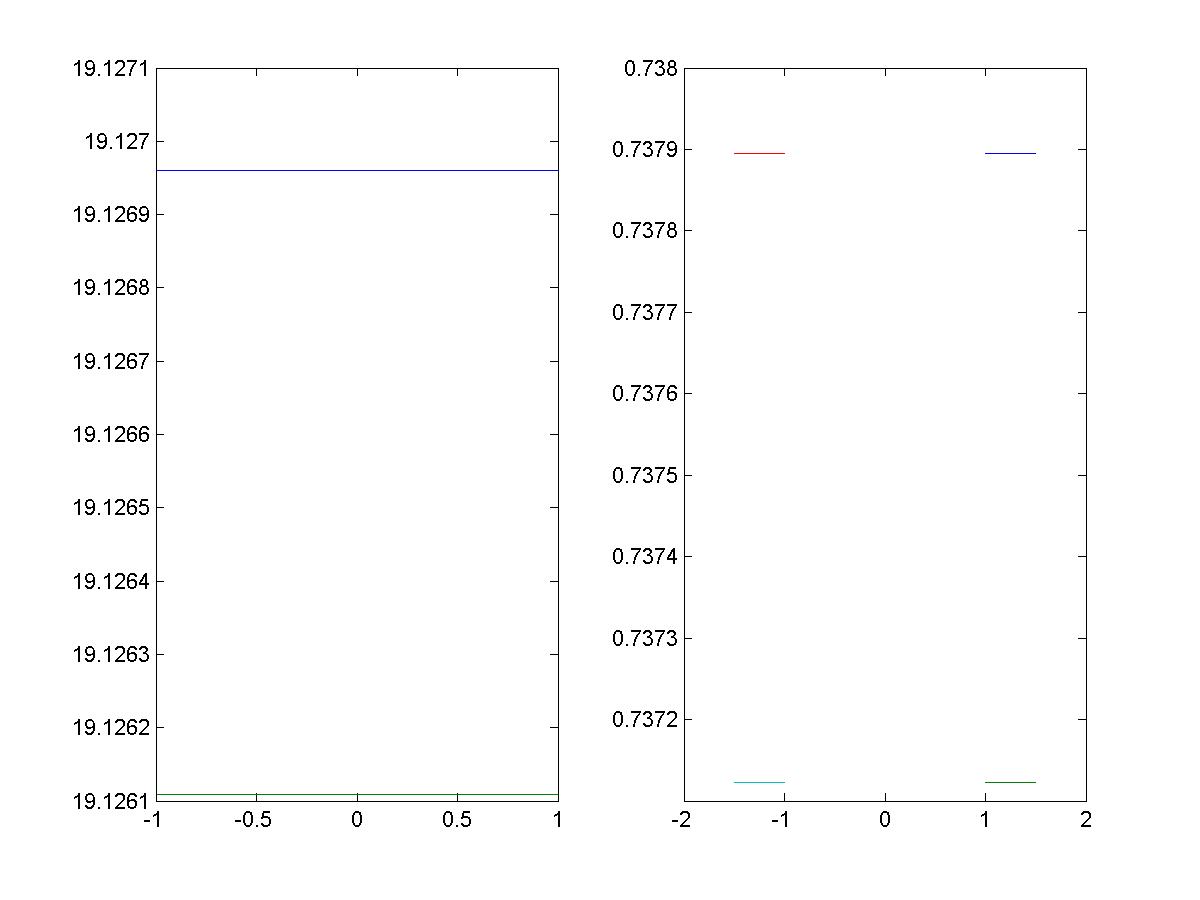}
\caption{The solution Max-Mixed interation and Min-Mixed interation (from top down ),$u(left),v(right)$, $N=100$, $\beta=6$}\label{fig8}
\end{figure}

Next, we only increase the value of $a$ so that we observe the change of $\beta_{c}$. To do that, we choose $a=1.1$.
\begin{figure}[H]
\centering
\includegraphics[height=0.46\textwidth, width=\textwidth]{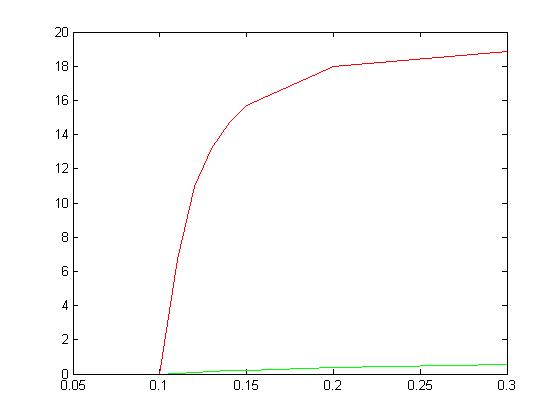}
\caption{Min-Mixed interation ($a=1.1$, $\beta_{c}\approx0.095$), $N=50$, $(u(top), v(bottom))$ vs $\beta$}\label{fig9}
\end{figure}

Comparing Fig.\ref{fig6} with Fig.\ref{fig9}, we see that $\beta_{c}$ decreases as $a$ becomes bigger. In fact, we can show the above fact is right numerically from Fig. \ref{fig10}, which fits the physical phenomenon very well.

\begin{figure}[H]
\centering
\includegraphics[height=0.5\textwidth, width=\textwidth]{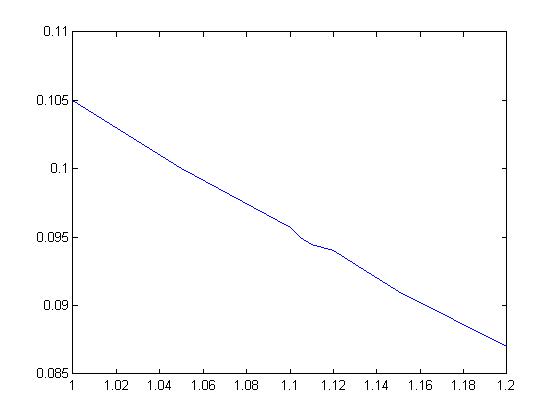}
\caption{Min-Mixed interation $N=50$, $\beta_{c}$ vs $a$}\label{fig10}
\end{figure}

Now, we only change the value of $K_{2}$ to observe the change of $\beta_{c}$.
\begin{figure}[H]
\centering
\includegraphics[height=0.5\textwidth, width=\textwidth]{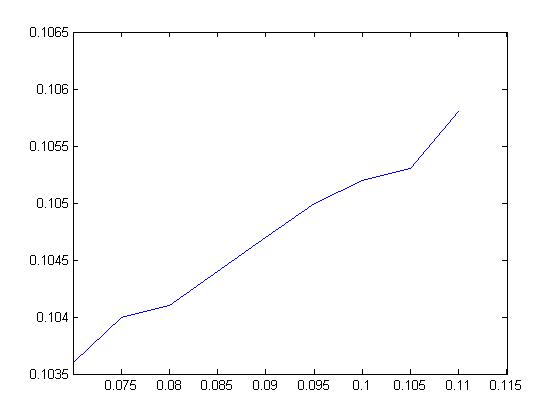}
\caption{Min-Mixed interation $N=50$, $\beta_{c}$ vs $K_{2}$}\label{fig11}
\end{figure}

From Fig.\ref{fig11}, we find that $\beta_{c}$ increases as $K_{2}$ increases.\\
{\bf Case 2}: $K_{1}\leq K_{2}$.
 Taking
\begin{equation}\begin{split}
a=0.5,\ b=1.5,\\
K_{1}=0.01, K_{2}=0.1.\\
\end{split}
\end{equation}
\begin{figure}[H]
\centering
\includegraphics[height=0.5\textwidth, width=\textwidth]{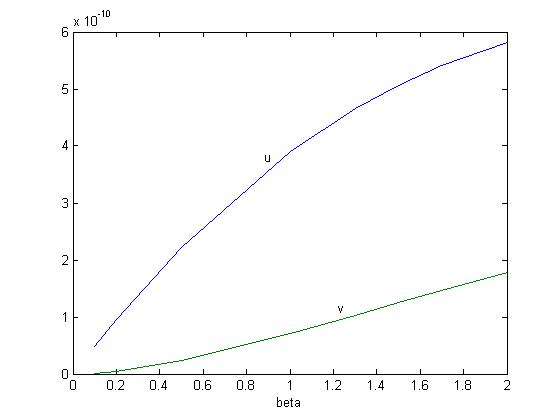}
\caption{Min-Mixed interation($N=50$, $(u(up), v(down))_m$ vs $\beta$}\label{fig12}
\end{figure}
\begin{figure}[H]
\centering
\includegraphics[height=0.45\textwidth, width=\textwidth]{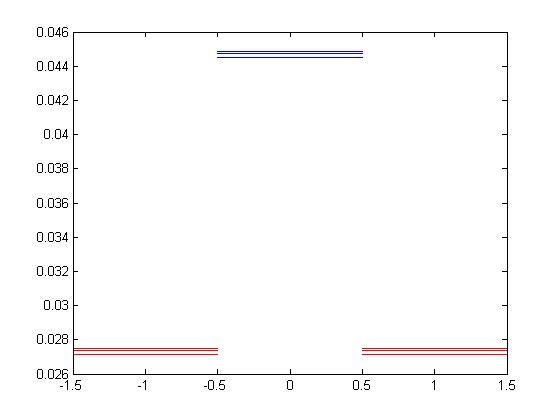}
\caption{Min-Mixed interation(from bottom to up), $u(middle),v(left , right),\beta=5, N=50, 100, 200$}\label{fig13}
\end{figure}
\begin{figure}[H]
\centering
\includegraphics[height=0.45\textwidth, width=\textwidth]{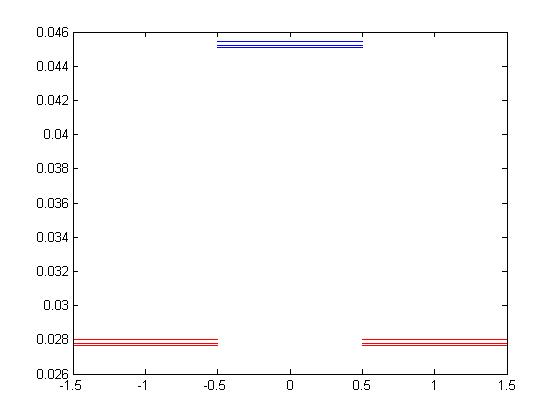}
\caption{Max-Mixed interation(from up to down),$u(middle),v(left , right),\beta=5, N=50, 100, 200$}\label{fig14}
\end{figure}
\begin{figure}[H]
\centering
\includegraphics[height=0.5\textwidth, width=\textwidth]{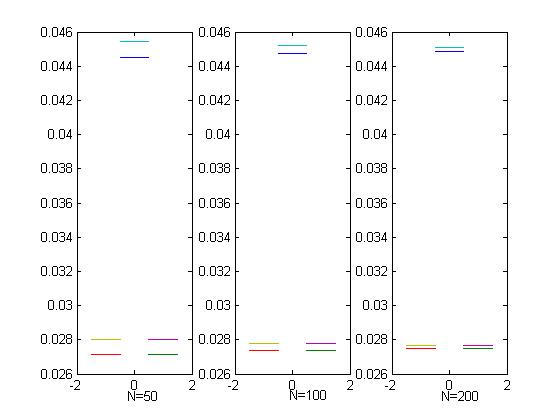}
\caption{The solution of Max-Mixed interation vs Min-Mixed interation(from top down ), $N=50,100,200, \beta=5$}\label{fig15}
\end{figure}
Next, we only increase the value of $a$ so that we observe the change of $(u,v)_{m}$.
\begin{figure}[H]
\centering
\includegraphics[height=0.45\textwidth, width=\textwidth]{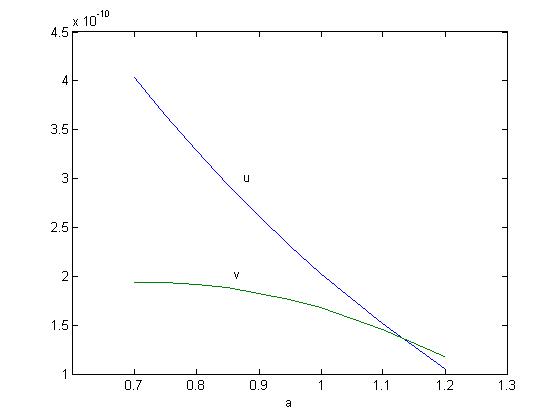}
\caption{Min-Mixed interation for $N=50$, $(u, v)_m$ vs $a$}\label{fig16}
\end{figure}
From Fig.\ref{fig16}, we find that $u$ and $v$ of the Min-Mixed interative scheme decrease with $a$ closing to $b$. Thus, $\beta_{c}$ will increase as $a$ close to $b$.\\

We now just change $K_{2}$ to observe the change of $(u,v)_{m}$.
\begin{figure}[H]
\centering
\includegraphics[height=0.45\textwidth, width=\textwidth]{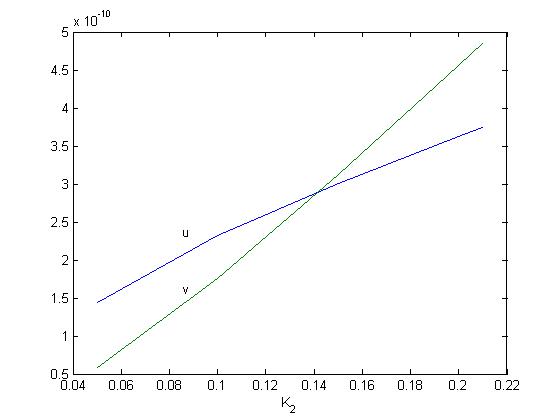}
\caption{Min-Mixed interation $N=50$, $(u,v)_{m}$ vs $K_{2}$}\label{fig17}
\end{figure}
From Fig.\ref{fig17}, we can see that $(u,v)_{m}$ increases as $K_{2}$ increases. Namely, $\beta_{c}$ will decrease as $K_{2}$ increases.\\



\noindent{\bf Acknowledgments.} {The work is supported by the Natural Science Foundation of China(No. 10901047).}

\end{document}